\documentclass[a4paper,11pt]{amsart}

\usepackage{graphicx}
\usepackage{mathptmx}
\usepackage{amsmath}
\usepackage{amssymb}
\usepackage{enumitem}
\usepackage{xcolor}

\newmuskip\pFqmuskip

\newcommand*\pFq[6][8]{%
  \begingroup 
  \pFqmuskip=#1mu\relax
  \mathcode`=\string"8000
  \begingroup\lccode`\~=`\,
  \lowercase{\endgroup\let~}\pFqcomma
  F^{#2}_{#3}{\left(\genfrac..{0pt}{}{#4}{#5}\bigg|#6\right)}%
  \endgroup
}
\newcommand{\pFqcomma}{\mskip\pFqmuskip}

\newtheorem{theorem}{Theorem}[section]
\newtheorem{lemma}[theorem]{Lemma}
\newtheorem{corollary}[theorem]{Corollary}
\newtheorem{proposition}[theorem]{Proposition}

\begin{document}

\title[]{Probabilistic degenerate logarithm and heterogeneous Stirling numbers}

\author{Dae San  Kim}
\address{Department of Mathematics, Sogang University, Seoul 121-742, Republic of Korea}
\email{dskim@sogang.ac.kr}
\author{Taekyun  Kim}
\address{Department of Mathematics, Kwangwoon University, Seoul 139-701, Republic of Korea}
\email{tkkim@kw.ac.kr}


\begin{abstract} 
Let $Y$ be a random variable whose moment-generating function exists in some neighborhood of the origin. While probabilistic Stirling numbers of the first and second kind have been introduced, early definitions often failed to satisfy fundamental orthogonality and inverse relations or lacked consistency with classical forms in the case when $Y=1$. This paper addresses these limitations by utilizing redefined probabilistic Stirling numbers, $S_{1}^{Y}(n,k)$ and $S_{2}^{Y}(n,k)$, alongside their degenerate counterparts. Our primary objective is twofold: first, to introduce the probabilistic (degenerate) logarithm associated with $Y$, providing explicit expressions for various random variables and defining new probabilistic degenerate Daehee and Cauchy numbers; and second, to investigate probabilistic heterogeneous Stirling numbers and establish a probabilistic degenerate version of the Schl\"omilch formula, demonstrating that these new frameworks maintain the essential algebraic properties of their classical counterparts.
\end{abstract}

\maketitle

\markboth{\centerline{\scriptsize Probabilistic degenerate logarithm and heterogeneous Stirling numbers}}
{\centerline{\scriptsize Dae San Kim and Taekyun Kim}}

\section{Preliminaries} 
The falling and the rising factorial sequences are respectively given by
\begin{align*}
&(x)_{0}=1,\quad (x)_{n}=x(x-1)\cdots(x-n+1),\quad(n \ge 1), \\
&\langle x \rangle_{0}=1,\quad \langle x \rangle_{n}=x(x+1)\cdots(x+n-1),\quad(n \ge 1).
\end{align*}
Let $\lambda$ be any nonzero real number. Then
the degenerate falling and rising factorial sequences are defined by
\begin{align*}
&(x)_{0,\lambda}=1,\quad (x)_{n,\lambda}=x(x-\lambda)\cdots(x-(n-1)\lambda),\quad(n \ge 1), \\
&\langle x \rangle_{0,\lambda}=1,\quad \langle x \rangle_{n,\lambda}=x(x+\lambda)\cdots(x+(n-1)\lambda),\quad(n \ge 1).
\end{align*}
The degenerate exponentials are introduced as
\begin{align*}
&e_{\lambda}^{x}(t)=(1+\lambda t)^{\frac{x}{\lambda}}=\sum_{n=0}^{\infty}(x)_{n,\lambda}\frac{t^{n}}{n!}, \\
&e_{\lambda}(t)=e_{\lambda}^{1}(t)=(1+\lambda t)^{\frac{1}{\lambda}}=\sum_{n=0}^{\infty}(1)_{n,\lambda}\frac{t^{n}}{n!}.
\end{align*} \par
The degenerate Stirling numbers of the second kind are equivalently defined as (see [12])
\begin{equation} \label{1}
\begin{aligned}
&\frac{1}{k!}\big(e_{\lambda}(t)-1 \big)^{k}=\sum_{n=k}^{\infty}S_{2,\lambda}(n,k)\frac{t^{n}}{n!},\quad (k \ge 0), \\
&(x)_{n,\lambda}=\sum_{k=0}^{n}S_{2,\lambda}(n,k)(x)_{k},\quad(n \ge 0), \\
&S_{2,\lambda}(n,k)=\frac{1}{k!}\sum_{j=0}^{k}\binom{k}{j}(-1)^{k-j}(j)_{n,\lambda},\quad(n \ge k).
\end{aligned}
\end{equation}
The degenerate Stirling numbers of the first kind are variously specified as (see [11])
\begin{equation} \label{2}
\begin{aligned}
&\frac{1}{k!}\big(\log_{\lambda}(1+t)\big)^{k}=\sum_{n=k}^{\infty}S_{1,\lambda}(n,k)\frac{t^{n}}{n!}, \quad (k \ge 0),\\
&(x)_{n}=\sum_{k=0}^{n}S_{1,\lambda}(n,k)(x)_{k,\lambda},\quad (n \ge 0), 
\end{aligned}
\end{equation}
where $\log_{\lambda}t=\frac{1}{\lambda}(t^{\lambda}-1)$ is called the degenerate logarithm and $\log_{\lambda}(1+t)$ is the compositional inverse of $e_{\lambda}(t)-1$. \\
Taking the limits $\lambda \rightarrow 0$, we find that
\begin{align*}
&(x)_{n,\lambda} \rightarrow x^{n}, \quad \langle x \rangle_{n,\lambda} \rightarrow x^{n}, \quad e_{\lambda}^{x}(t) \rightarrow e^{xt}, \quad e_{\lambda}(t) \rightarrow e^{t},\\
&\log_{\lambda}t \rightarrow \log t, \quad S_{1,\lambda}(n,k) \rightarrow S_{1}(n,k), \quad 
S_{2,\lambda}(n,k) \rightarrow S_{2}(n,k),
\end{align*}
where $S_{1}(n,k)$ and $S_{2}(n,k)$ are respectively the Stirling numbers of the first kind and those of the second kind characterized by
\begin{equation} \label{3}
\frac{1}{k!}(e^{t}-1)^{k}=\sum_{n=k}^{\infty}S_{2}(n,k)\frac{t^{n}}{n!}, \quad
\frac{1}{k!}\big(\log(1+t)\big)^{k}=\sum_{n=k}^{\infty}S_{1}(n,k)\frac{t^{n}}{n!},\quad (k \ge 0).
\end{equation}
Furthermore, taking the limits $\lambda \rightarrow 1$ yields
\begin{equation*}
(x)_{n,\lambda} \rightarrow (x)_{n}, \quad \langle x \rangle_{n, \lambda} \rightarrow \langle x \rangle_{n}.
\end{equation*} \par
Recall that a power series $f(t)=\sum_{n=0}^{\infty}a_{n}\frac{t^{n}}{n!}$ is a delta series if $a_{0}=0$ and $a_{1} \ne 0$.
We recall three of the Lagrange inversion formula and its variants: for any formal power series $g(t)$ and any delta series $f(t)$, they are stated as (see [6,24]):
\begin{flalign*}
&(A)\,\,\,[t^{n}]\,g(\bar{f}(t))=\frac{1}{n}[t^{n-1}]\,g^{\prime}(t)\Big(\frac{t}{f(t)}\Big)^{n}, \\
&(B)\,\,\,[t^{n}]\,\big(\bar{f}(t)\big)^{k}=\frac{k}{n}[t^{n-k}]\, \Big(\frac{t}{f(t)}\Big)^{n}, \\
&(C)\,\,\,[t^{n}]\,\bar{f}(t)=\frac{1}{n}[t^{n-1}]\,\Big(\frac{t}{f(t)}\Big)^{n}, &&
\end{flalign*}
where $\bar{f}(t)$ is the compositional inverse of $f(t)$, $n$ is any nonnegative integer and $k$ is any positive integer. \\
Applying the formula (C) to $f(t)=e_{\lambda}(t)-1$, we arrive at
\begin{equation} \label{4}
[t^{n}]\, \log_{\lambda}(1+t)=\frac{1}{n}[t^{n-1}]\,\Big(\frac{t}{e_{\lambda}(t)-1} \Big)^{n}=\frac{1}{n!}\beta_{n-1,\lambda}^{(n)},
\end{equation}
where the degenerate Bernoulli polynomials and numbers of order $\gamma$ are respectively given by
\begin{equation} \label{5}
\Big(\frac{t}{e_{\lambda}(t)-1}\Big)^{\gamma}e_{\lambda}^{x}(t)=\sum_{n=0}^{\infty}\beta_{n,\lambda}^{(\gamma)}(x)\frac{t^{n}}{n!},\quad \beta_{n,\lambda}^{(\gamma)}=\beta_{n,\lambda}^{(\gamma)}(0),
\end{equation}
for any complex number $\gamma$.
We see from \eqref{4} that
\begin{equation} \label{6}
\log_{\lambda}(1+t)=\sum_{n=1}^{\infty}\beta_{n-1,\lambda}^{(n)}\frac{t^{n}}{n!}.
\end{equation}
In addition, applying the formula (B) to $f(t)=e_{\lambda}(t)-1$ gives
\begin{equation} \label{7}
[t^{n}]\,\big(\log_{\lambda}(1+t)\big)^{k}=\frac{k}{n}\beta_{n-k,\lambda}^{(n)}\frac{1}{(n-k)!}.
\end{equation}
Thus it follows from \eqref{7} that 
\begin{equation} \label{8}
\frac{1}{k!}\big(\log_{\lambda}(1+t)\big)^{k}=\sum_{n=k}^{\infty}\binom{n-1}{k-1}\beta_{n-k,\lambda}^{(n)}\frac{t^{n}}{n!},\quad (k \ge 0).
\end{equation}
Now, one observes from \eqref{2} and \eqref{8} that
\begin{equation} \label{9}
S_{1,\lambda}(n,k)=\binom{n-1}{k-1}\beta_{n-k,\lambda}^{(n)}, \quad (n \ge k).
\end{equation}
Taking the limits $\lambda \rightarrow 0$ in \eqref{6} and \eqref{9} yields
\begin{equation*}
B_{n-1}^{(n)}=(-1)^{n-1}(n-1)!,\quad S_{1}(n,k)=\binom{n-1}{k-1}B_{n-k}^{(n)},
\end{equation*}
where the Bernoulli polynomials and numbers of order $\gamma$ are respectively defined by
\begin{equation} \label{10}
\Big(\frac{t}{e^{t}-1} \Big)^{\gamma}e^{x t}=\sum_{n=0}^{\infty}B_{n}^{(\gamma)}(x)\frac{t^{n}}{n!},\quad B_{n}^{(\gamma)}=B_{n}^{(\gamma)}(0).
\end{equation} \par
The Lah numbers are equivalently specified as
\begin{equation} \label{11}
\begin{aligned}
&\frac{1}{k!}\Big(\frac{t}{1-t}\Big)^{k}=\sum_{n=k}^{\infty}L(n,k)\frac{t^{n}}{n!}, \quad (k \ge 0),\quad L(n,k)=\frac{n!}{k!}\binom{n-1}{k-1},\quad (n \ge k), \\
&\langle x \rangle_{n}=\sum_{k=0}^{n}L(n,k)(x)_{k}, \quad (n \ge 0).
\end{aligned}
\end{equation}
The Lah-Bell polynomials are variously characterized by
\begin{equation} \label{12}
\begin{aligned}
&LB_{n}(x)=\sum_{k=0}^{n}L(n,k)x^{k}, \quad
e^{x\big(\frac{1}{1-t}-1\big)}=e^{x\frac{t}{1-t}}=\sum_{n=0}^{\infty}LB_{n}(x)\frac{t^{n}}{n!}, \\
&LB_{n}(x)=e^{-x}\sum_{k=0}^{n}\frac{\langle k \rangle_{n}}{k!} x^{k}.
\end{aligned}
\end{equation} \par
The heterogeneous Stirling numbers of the second kind are equivalently introduced as  (see [17])
\begin{equation} \label{13}
\begin{aligned}
&\langle x \rangle_{n,\lambda}=\sum_{k=0}^{n}H_{\lambda}(n,k)(x)_{k},  \quad (n \ge 0),\\
&\frac{1}{k!}\big(e_{-\lambda}(t)-1\big)^{k}=\frac{1}{k!}\big(e_{\lambda}^{-1}(-t)-1\big)^{k}=\sum_{n=k}^{\infty}H_{\lambda}(n,k)\frac{t^{n}}{n!}, ,\quad (k \ge 0),\\
&H_{\lambda}(n,k)=\frac{1}{k!}\sum_{j=0}^{k}\binom{k}{j}(-1)^{k-j}\langle j \rangle_{n,\lambda}, \quad(n \ge k),
\end{aligned}
\end{equation}
which show 
\begin{equation*} 
H_{\lambda}(n,k)=S_{2,-\lambda}(n,k)=(-1)^{n}S_{2,\lambda}^{-1}(n,k),\quad (n \ge k),
\end{equation*}
with $\frac{1}{k!}\big(e_{\lambda}^{x}(t)-1\big)^{k}=\sum_{n=k}^{\infty}S_{2,\lambda}^{x}(n,k)\frac{t^{n}}{n!}$. \\
Here we see  from \eqref{3}, \eqref{11} and \eqref{13} that 
\begin{equation*}
\lim_{\lambda \rightarrow 0}H_{\lambda}(n,k)=S_{2}(n,k), \quad
\lim_{\lambda \rightarrow 1}H_{\lambda}(n,k)=L(n,k).
\end{equation*}
The heterogeneous Stirling numbers of the first kind are defined, by inverting the relations in \eqref{13}, as (see [17])
\begin{equation} \label{14}
\begin{aligned}
&(x)_{n}=\sum_{k=0}^{n}G_{\lambda}(n,k)\langle x \rangle_{k,\lambda}, \quad (n \ge 0),\\
&\frac{1}{k!}\Big(-\log_{\lambda}\Big(\frac{1}{1+t}\Big)\Big)^{k}=\frac{1}{k!}\Big(\log_{-\lambda}\big(1+t \big)\Big)^{k}=\sum_{n=k}^{\infty}G_{\lambda}(n,k)\frac{t^{n}}{n!},\quad (k \ge 0).
\end{aligned}
\end{equation}
From \eqref{14} and the formula (B), we find that
\begin{equation*}
G_{\lambda}(n,k)=\binom{n-1}{k-1}\beta_{n-k, -\lambda}^{(n)}, \quad (n \ge k).
\end{equation*}
Applying the formula (C) to $f(t)=\log_{-\lambda}(1+t)$ gives
\begin{equation*}
[t^{n}]\,\big(e_{-\lambda}(t)-1\big)^{k}=\frac{k}{n}[t^{n-k}]\,\Big(\frac{t}{\log_{-\lambda}(1+t)}\Big)^{n}=\frac{k}{n}C_{n-k,-\lambda}^{(n)} \frac{1}{(n-k)!},
\end{equation*}
which shows
\begin{equation} \label{15}
\frac{1}{k!}\big(e_{-\lambda}(t)-1\big)^{k}=\sum_{n=k}^{\infty}\binom{n-1}{k-1}C_{n-k, -\lambda}^{(n)}\frac{t^{n}}{n!},\quad (k \ge 0).
\end{equation}
Thus, from \eqref{1}, \eqref{2}, \eqref{13} and \eqref{15}, we arrive at
\begin{equation*}
H_{\lambda}(n,k)=\binom{n-1}{k-1}C_{n-k, -\lambda}^{(n)}, \quad S_{2,\lambda}(n,k)=\binom{n-1}{k-1}C_{n-k, \lambda}^{(n)}, \quad (n \ge k).
\end{equation*}
Here the degenerate Cauchy numbers of order $\gamma$ are given by
\begin{equation} \label{16}
\Big(\frac{t}{\log_{\lambda}(1+t)}\Big)^{\gamma}=\sum_{n=0}^{\infty}C_{n,\lambda}^{(\gamma)}\frac{t^{n}}{n!}.
\end{equation} \par
Heterogeneous Bell polynomials and numbers are respectively defined by
\begin{equation} \label{17}
H_{n,\lambda}(x)=\sum_{k=0}^{n}H_{\lambda}(n,k)x^{k}, \quad H_{n,\lambda}=H_{n,\lambda}(1).
\end{equation}
This entails that
\begin{equation} \label{18}
\begin{aligned}
&e^{x\big(e_{-\lambda}(t)-1\big)}=\sum_{n=0}^{\infty}H_{n,\lambda}(x)\frac{t^{n}}{n!}, \\
&H_{n,\lambda}(x)=e^{-x}\sum_{k=0}^{\infty}\frac{\langle k \rangle_{n,\lambda}}{k!}x^{k}.
\end{aligned}
\end{equation}
Moreover, one shows that
\begin{align*}
&H_{\lambda}(n,k)=\sum_{l=k}^{n}(-1)^{n-l}S_{2}(l,k)S_{1}(n,l)\lambda^{n-l}, \\
&L(n,k)=\sum_{l=k}^{n}(-1)^{n-l}S_{1,\lambda}(n,l)H_{\lambda}(l,k),\quad (n \ge k).
\end{align*} \par
The partial Bell polynomials are characterized by
\begin{equation*}
B_{n,k}(x_{1},x_{2},\dots,x_{n-k+1})=\sum_{\substack{j_{1}+j_{2}+\cdots=k \\ j_{1}+2j_{2}+\dots=n}}\frac{n!}{j_{1}!j_{2}!\cdots}\Big(\frac{x_{1}}{1!} \Big)^{j_{1}}\Big(\frac{x_{2}}{2!}\Big)^{j_{2}}\cdots,
\end{equation*}
where $j_{1}, j_{2}, \dots$ are nonnegative integers. The generating function of the partial Bell polynomials is specified as
\begin{equation} \label{19}
\frac{1}{k!}\bigg(\sum_{m=1}^{\infty}x_{m}\frac{t^{m}}{m!}\bigg)^{k}=\sum_{n=k}^{\infty}B_{n,k}(x_{1},x_{2},\dots,x_{n-k+1})\frac{t^{n}}{n!}, \quad (k \ge 0).
\end{equation}
General references for this paper include [6,23,24,25].

\section{Introduction}
Throughout this paper, we assume that $Y$ is a random variable such that the moment-generating function of $Y$
\begin{equation} \label{20}
E[e^{tY}]=\sum_{n=0}^{\infty}E[Y^{n}]\frac{t^{n}}{n!}, \quad (|t|<r)
\end{equation}
exists for some $r > 0$, where $E$ is the mathematical expectation. We assume further that $E[Y] \ne 0$. Let $(Y_{k})_{k\ge 1}$ be a sequence of independent and identically distributed (i.i.d.) copies of $Y$, and let $S_{k}$ denote their partial sums:
\begin{equation} \label{21}
S_{k}=Y_{1}+\cdots+Y_{k}, \quad (k \ge 1), \quad S_{0}=0.
\end{equation}
The probabilisitc degenerate Bernoulli polynomials and numbers of order $\gamma$ associated with $Y$ are respectively given by
\begin{equation} \label{22}
\bigg(\frac{t}{E[e_{\lambda}^{Y}(t)]-1}\bigg)^{\gamma}\big(E[e_{\lambda}^{Y}(t)]\big)^{x}=\sum_{n=0}^{\infty}\beta_{n,\lambda}^{(\gamma,Y)}(x)\frac{t^{n}}{n!},\quad \beta_{n,\lambda}^{(\gamma,Y)}=\beta_{n,\lambda}^{(\gamma,Y)}(0).
\end{equation} \par
The probabilistic degenerate Stirling numbers of the second kind associated with $Y$ are equivalently defined as
\begin{equation} \label{23}
\begin{aligned}
&\frac{1}{k!}\big(E[e_{\lambda}^{Y}(t)]-1 \big)^{k}=\sum_{n=k}^{\infty}S_{2,\lambda}^{Y}(n,k)\frac{t^{n}}{n!}, \quad (k \ge 0),\\
&S_{2,\lambda}^{Y}(n,k)=\frac{1}{k!}\sum_{j=0}^{k}\binom{k}{j}(-1)^{k-j}E[(S_{j})_{n,\lambda}], \quad (n \ge k),\\
&S_{2,\lambda}^{Y}(n,k)=B_{n,k}\big(E[(Y)_{1,\lambda}],E[(Y)_{2,\lambda}],\dots,E[(Y)_{n-k+1,\lambda}]\big), \quad (n \ge k).
\end{aligned}
\end{equation}
Let $e_{Y,\lambda}(t)=E[e_{\lambda}^{Y}(t)]-1$. Then the probabilistic degenerate Stirling numbers of the first kind associated with $Y$ are introduced as
\begin{equation} \label{24}
\frac{1}{k!}\big(\bar{e}_{Y,\lambda}(t)\big)^{k}=\sum_{n=k}^{\infty}S_{1,\lambda}^{Y}(n,k)\frac{t^{n}}{n!},\quad (k \ge 0),
\end{equation}
where $\bar{e}_{Y,\lambda}(t)$ is the compositional inverse of $e_{Y,\lambda}(t)$.
Among other things, from \eqref{23} and \eqref{24}, we note that $S_{1,\lambda}^{Y}(n,k)$ and $S_{2,\lambda}^{Y}(n,k)$ satisfy the orthogonality and inverse relations.
\begin{proposition}
The following orthogonality and inverse relations are valid for $S_{1,\lambda}^{Y}(n,k)$ and $S_{2,\lambda}^{Y}(n,k)$.
\begin{flalign*}
&(a)\,\, \,\sum_{k=l}^{n} S_{2,\lambda}^{Y}(n,k)S_{1,\lambda}^{Y}(k,l)=\delta_{n,l}, \quad \sum_{k=l}^{n} S_{1,\lambda}^{Y}(n,k)S_{2,\lambda}^{Y}(k,l)=\delta_{n,l}, \\
&(b)\,\, a_{n}=\sum_{k=0}^{n}S_{2,\lambda}^{Y}(n,k) b_{k}\,\, \iff \,\, b_{n}=\sum_{k=0}^{n}S_{1,\lambda}^{Y}(n,k)a_{k}, \\ 
&(c)\,\, a_{n}=\sum_{k=n}^{m}S_{2,\lambda}^{Y}(k,n)b_{k} \,\, \iff \,\, b_{n}=\sum_{k=n}^{m}S_{1,\lambda}^{Y}(k,n)a_{k}. &&
\end{flalign*}
\end{proposition}
By using the formulas (B) and (C), \eqref{22} and \eqref{24}, we show that
\begin{equation} \label{25}
\begin{aligned}
&\bar{e}_{Y,\lambda}(t)=\sum_{n=1}^{\infty}\beta_{n-1,\lambda}^{(n,Y)}\frac{t^{n}}{n!}, \\
&\frac{1}{k!}\big(\bar{e}_{Y,\lambda}(t)\big)^{k}=\sum_{n=k}^{\infty}S_{1,\lambda}^{Y}(n,k)\frac{t^{n}}{n!}=\sum_{n=k}^{\infty}\binom{n-1}{k-1}\beta_{n-k,\lambda}^{(n,Y)}\frac{t^{n}}{n!},\quad (k \ge 0).
\end{aligned}
\end{equation}
Taking the limits $\lambda \rightarrow 0$ in \eqref{25} and for $e_{Y}(t)=E[e^{Y t}]-1$, we see that
\begin{equation}\label{26}
\begin{aligned}
&\bar{e}_{Y}(t)=\sum_{n=1}^{\infty}B_{n-1}^{(n,Y)}\frac{t^{n}}{n!}, \\
&\frac{1}{k!}\big(\bar{e}_{Y}(t)\big)^{k}=\sum_{n=k}^{\infty}S_{1}^{Y}(n,k)\frac{t^{n}}{n!}=\sum_{n=k}^{\infty}\binom{n-1}{k-1}B_{n-k}^{(n,Y)}\frac{t^{n}}{n!},\quad (k \ge 0).
\end{aligned}
\end{equation}
Here the probabilisitc Bernoulli polynomials and numbers of order $\gamma$ associated with $Y$ are respectively characterized by
\begin{equation} \label{27}
\bigg(\frac{t}{E[e^{Y t}]-1}\bigg)^{\gamma}\big(E[e^{Y t}]\big)^{x}=\sum_{n=0}^{\infty}B_{n}^{(\gamma,Y)}(x)\frac{t^{n}}{n!},\quad B_{n}^{(\gamma,Y)}=B_{n}^{(\gamma,Y)}(0).
\end{equation}
The following theorem follows from \eqref{19} and \eqref{25}.
\begin{theorem}
For any integers $n \ge k \ge 0$, we have
\begin{equation*}
S_{1,\lambda}^{Y}(n,k)=\binom{n-1}{k-1}\beta_{n-k,\lambda}^{(n,Y)}=B_{n,k}\Big(\beta_{0,\lambda}^{(1,Y)},\beta_{1,\lambda}^{(2,Y)}, \dots \beta_{n-k,\lambda}^{(n-k+1,Y)} \Big).
\end{equation*}
\end{theorem}
In recent years, various degenerate versions of many special numbers and polynomials have been studied, yielding significant identities and combinatorial properties (see [9,11,12,17,19] and the references therein). 
For example, we found the degenerate Stirling numbers of the first kind and the second kind which turned out to be very important in studying degenerate versions of special polynomials and numbers. This exploration of degenerate versions originated with Carlitz's work on degenerate Bernoulli and Euler polynomials (see [5]). Notably, this area of study has expanded beyond polynomials and numbers to encompass transcendental functions (see [13,14]) and umbral calculus (see [8]). \par
Concurrently, probabilistic extensions of special polynomials and numbers have emerged as a vibrant intersection of combinatorics, probability theory, and mathematical physics (see [4,7,15,16,18,20,21,26]). Driven by the need to model complex stochastic systems, mathematicians have increasingly focused on ``probabilistic" versions of classical sequences, such as the Bernoulli, Euler, Bell, and Stirling polynomials, by associating them with specific random variables $Y$ through their moment-generating functions. \par
Recently, Adell and Lekuona introduced the probabilistic Stirling numbers of the second kind associated with a random variable $Y$, denoted by $S_{2}^{Y}(n,k)$ (see [1,3,20]). Subsequently, Adell and B\'{e}nyi defined the probabilistic Stirling numbers of the first kind associated with $Y$, $s_{Y}(n,k)$, by employing the cumulant generating function (see [2]). However, these versions of $S_{2}^{Y}(n,k)$ and $s_{Y}(n,k)$ fail to satisfy the standard orthogonality, limiting their utility in inversion formulas. Furthermore, $s_{Y}(n,k)$ does not reduce to the classical Stirling numbers of the first kind, $S_{1}(n,k)$, in the case where $Y=1$. To address these shortcomings, a redefined version of the probabilistic Stirling numbers of the first kind, denoted $S_{1}^{Y}(n,k)$, was introduced in [7,26]. These redefined numbers, together with $S_{2}^{Y}(n,k)$, satisfy the expected orthogonality and inverse relations, and $S_{1}^{Y}(n,k)$ correctly reduces to $S_{1}(n,k)$ when $Y=1$. Similarly, their degenerate counterparts--the probabilistic degenerate Stirling numbers of the second kind $S_{2,\lambda}^{Y}(n,k)$ and of the first kind $S_{1,\lambda}^{Y}(n,k)$--were considered in [7]. These also satisfy orthogonality and inverse relations (see Proposition 1.1), with $S_{1,\lambda}^{Y}(n,k)$ reducing to the degenerate Stirling numbers of the first kind $S_{1,\lambda}(n,k)$ when $Y=1$. We note that an alternative definition for $S_{1,\lambda}^{Y}(n,k)$ was proposed in [16] using the degenerate cumulant generating function; however, that version does not satisfy the orthogonality and inverse relations when paired with $S_{2,\lambda}^{Y}(n,k)$. \par
The aim of this paper is twofold. First, we define the probabilistic logarithm associated with $Y$, $\log^{Y}(1+t)$, and the probabilistic degenerate logarithm associated with $Y$, $\log_{\lambda}^{Y}(1+t)$. These are defined such that for any integer $k \ge 0$, the following generating functions hold:
\begin{equation*}
\frac{1}{k!}\Big(\log_{\lambda}^{Y}(1+t)\Big)^{k}=\sum_{n=k}^{\infty}S_{1,\lambda}^{Y}(n,k)\frac{t^{n}}{n!}, \,\,\, \frac{1}{k!}\Big(\log^{Y}(1+t)\Big)^{k}=\sum_{n=k}^{\infty}S_{1}^{Y}(n,k)\frac{t^{n}}{n!},\,\, (k \ge 0).
\end{equation*}
Specifically, for $k=1$, our definitions for the probabilistic logarithms (see \eqref{42}) are:
\begin{equation*}
\log_{\lambda}^{Y}(1+t)=\sum_{n=1}^{\infty}S_{1,\lambda}^{Y}(n,1)\frac{t^{n}}{n!}, \quad \log^{Y}(1+t)=\sum_{n=1}^{\infty}S_{1}^{Y}(n,1)\frac{t^{n}}{n!}.
\end{equation*}
In Section 5, we provide explicit expressions for these logarithms for various discrete and continuous random variables. Furthermore, using our definition of the probabilistic degenerate logarithm, we introduce the probabilistic degenerate Daehee numbers of order $\gamma$ associated with $Y$ (see \eqref{43}, \eqref{44}), $D_{n,\lambda}^{(\gamma,Y)}$, and the probabilistic degenerate Cauchy numbers of order $\gamma$ associated with $Y$ (see \eqref{45}, \eqref{46}), $C_{n,\lambda}^{(\gamma,Y)}$. For the case $\gamma=1$, the probabilistic degenerate Daehee numbers associated with $Y$ are defined by:
\begin{equation*}
\frac{\log_{\lambda}^{Y}(1+t)}{t}=\sum_{n=0}^{\infty}D_{n,\lambda}^{Y}\frac{t^{n}}{n!}.
\end{equation*}
We note that previous definitions for probabilistic degenerate Daehee polynomials $D_{n,\lambda}^{Y}(x)$ and $d_{n,\lambda}^{Y}(x)$ (see [22,27]) exist; however, they possess a significant limitation: for $x=0$, both $D_{n,\lambda}^{Y}(0)$ and $d_{n,\lambda}^{Y}(0)$ become independent of the random variable $Y$, thereby losing their probabilistic character.  \par
Second, we investigate the probabilistic heterogeneous Stirling numbers of the second kind, $H_{\lambda}^{Y}(n,k)$, and of the first kind, $G_{\lambda}^{Y}(n,k)$. These satisfy orthogonality and inverse relations, and reduce to $H_{\lambda}(n,k)$ and $G_{\lambda}(n,k)$ when $Y=1$. While a probabilistic version of the Schl\"omilch formula--expressing $S_{1}^{Y}(n,k)$ in terms of $S_{2}^{Y}(n,k)$--was shown in [26], we demonstrate that a probabilistic degenerate version also holds. Specifically, we show that the Schl\"omilch identity remains valid when $S_{1}^{Y}(n,k)$ and $S_{2}^{Y}(n-k+j,j)$ are replaced by $S_{1,\lambda}^{Y}(n,k)$ and $S_{2,\lambda}^{Y}(n-k+j,j)$, or by $G_{\lambda}^{Y}(n,k)$ and $H_{\lambda}^{Y}(n-k+j,j)$. This result further yields an expression for $\log_{\lambda}^{Y}(1+t)$ in terms of the probabilistic degenerate Stirling numbers of the second kind. While probabilistic heterogeneous Stirling numbers of the second kind are discussed in [18], those of the first kind have not yet been addressed. \par

\section{Probabilistic heterogeneous Stirling numbers}

The probabilistic heterogeneous Stirling numbers of the second kind associated with $Y$ are defined by
\begin{equation} \label{28}
\frac{1}{k!}\Big(E[e_{-\lambda}^{Y}(t)]-1\Big)^{k}=\frac{1}{k!}\Big(E[e_{\lambda}^{-Y}(-t)]-1\Big)^{k}=\sum_{n=k}^{\infty}H_{\lambda}^{Y}(n,k) \frac{t^{n}}{n!}, \,\,(k \ge 0),
\end{equation}
which implies that
\begin{equation} \label{29}
H_{\lambda}^{Y}(n,k)=(-1)^{n}S_{2,\lambda}^{-Y}(n,k)=S_{2,-\lambda}^{Y}(n,k), \quad (n \ge k).
\end{equation}
We observe that
\begin{equation} \label{30}
\begin{aligned}
\sum_{n=k}^{\infty}H_{\lambda}^{Y}(n,k) \frac{t^{n}}{n!}&=\frac{1}{k!}\Big(E[e_{-\lambda}^{Y}(t)]-1\Big)^{k}
=\frac{1}{k!}\sum_{j=0}^{k}\binom{k}{j}(-1)^{k-j}E[e_{-\lambda}^{S_{j}}(t)] \\
&=\frac{1}{k!}\sum_{j=0}^{k}\binom{k}{j}(-1)^{k-j}\sum_{n=0}^{\infty}E[(S_{j})_{n,-\lambda}]\frac{t^{n}}{n!} \\
&=\frac{1}{k!}\sum_{j=0}^{k}\binom{k}{j}(-1)^{k-j}\sum_{n=0}^{\infty}E[\langle S_{j}\rangle_{n,\lambda}]\frac{t^{n}}{n!} \\
&=\sum_{n=0}^{\infty}\frac{1}{k!}\sum_{j=0}^{k}\binom{k}{j}(-1)^{k-j}E[\langle S_{j}\rangle_{n,\lambda}]\frac{t^{n}}{n!}.
\end{aligned}
\end{equation}
Moreover, from \eqref{19}, we note
\begin{equation} \label{31}
\begin{aligned}
\sum_{n=k}^{\infty}H_{\lambda}^{Y}(n,k) \frac{t^{n}}{n!}&=\frac{1}{k!}\bigg(\sum_{m=1}^{\infty}E[\langle Y \rangle_{m,\lambda}]\frac{t^{m}}{m!}\bigg)^{k} \\
&=\sum_{n=k}^{\infty}B_{n,k}(E[\langle Y \rangle_{1,\lambda}],E[\langle Y \rangle_{2,\lambda}],\dots, E[\langle Y \rangle_{n-k+1,\lambda}])\frac{t^{n}}{n!}.
\end{aligned}
\end{equation}
Now, from \eqref{29}, \eqref{30} and \eqref{31}, we derive the following result.
\begin{theorem}
For any integers $n \ge k \ge 0$, we have
\begin{align*}
H_{\lambda}^{Y}(n,k)&=(-1)^{n}S_{2,\lambda}^{-Y}(n,k)=S_{2,-\lambda}^{Y}(n,k)\\
&=\frac{1}{k!}\sum_{j=0}^{k}\binom{k}{j}(-1)^{k-j}E[\langle S_{j}\rangle_{n,\lambda}] \\
&=B_{n,k}(E[\langle Y \rangle_{1,\lambda}],E[\langle Y \rangle_{2,\lambda}],\dots, E[\langle Y \rangle_{n-k+1,\lambda}]).
\end{align*}
\end{theorem}
We let
\begin{equation} \label{32}
\mathcal{E}_{Y,\lambda}(t)=E[e_{\lambda}^{-Y}(-t)]-1=E[e_{-\lambda}^{Y}(t)]-1. 
\end{equation}
The {\it{probabilistic heterogeneous Stirling numbers of the first kind associated with $Y$}} are specified as
\begin{equation} \label{33}
\frac{1}{k!}\Big(\bar{\mathcal{E}}_{Y,\lambda}(t)\Big)^{k}=\sum_{n=k}^{\infty}G_{\lambda}^{Y}(n,k)\frac{t^{n}}{n!},\quad (k \ge 0).
\end{equation}
By using the formula (C), \eqref{22} and the first expression in \eqref{32}, we derive that
\begin{equation} \label{34}
\begin{aligned}
[t^{n}]\,\bar{\mathcal{E}}_{Y,\lambda}(t)&=\frac{1}{n}(-1)^{n}[t^{n-1}]\,\bigg(\frac{-t}{E[e_{\lambda}^{-Y}(-t)]-1}\bigg)^{n} \\
&=\frac{1}{n}(-1)^{n}[t^{n-1}]\,\sum_{m=0}^{\infty}\beta_{m,\lambda}^{(n,-Y)}(-1)^{m}\frac{t^{m}}{m!}=-\beta_{n-1,\lambda}^{(n,-Y)}\frac{1}{n!}.
\end{aligned}
\end{equation}
Also, by using the formula (C), \eqref{22} and the second expression in \eqref{32}, we find that
\begin{equation} \label{35}
\begin{aligned}
[t^{n}]\,\bar{\mathcal{E}}_{Y,\lambda}(t)&=\frac{1}{n}[t^{n-1}]\,\bigg(\frac{t}{E[e_{-\lambda}^{Y}(t)]-1}\bigg)^{n} \\
&=\frac{1}{n}[t^{n-1}]\,\sum_{m=0}^{\infty}\beta_{m,-\lambda}^{(n,Y)}\frac{t^{m}}{m!}=\beta_{n-1,-\lambda}^{(n,Y)}\frac{1}{n!}.
\end{aligned}
\end{equation}
Thus, from \eqref{34} and \eqref{35}, we deduce that
\begin{equation} \label{36}
\bar{\mathcal{E}}_{Y,\lambda}(t)=-\sum_{n=1}^{\infty}\beta_{n-1,\lambda}^{(n,-Y)}\frac{t^{n}}{n!}=\sum_{n=1}^{\infty}\beta_{n-1,-\lambda}^{(n,Y)}\frac{t^{n}}{n!},
\end{equation}
and hence, from \eqref{36}, we obtain
\begin{equation} \label{37}
\begin{aligned}
\frac{1}{k!}\Big( \bar{\mathcal{E}}_{Y,\lambda}(t) \Big)^{k}&=\frac{1}{k!}\bigg(\sum_{m=1}^{\infty}-\beta_{m-1,\lambda}^{(m,-Y)}\frac{t^{m}}{m!} \bigg)^{k} \\
&=\frac{1}{k!}\bigg(\sum_{m=1}^{\infty}\beta_{m-1,-\lambda}^{(m,Y)}\frac{t^{m}}{m!} \bigg)^{k},\quad (k \ge 0).
\end{aligned}
\end{equation}
In addition, using the formula (B) and the first and second expressions in \eqref{32}, we show that
\begin{equation} \label{38}
\begin{aligned}
\frac{1}{k!}\Big( \bar{\mathcal{E}}_{Y,\lambda}(t) \Big)^{k}&=\sum_{n=k}^{\infty}\binom{n-1}{k-1}(-1)^{k}\beta_{n-k,\lambda}^{(n,-Y)}\frac{t^{n}}{n!} \\
&=\sum_{n=k}^{\infty}\binom{n-1}{k-1}\beta_{n-k,-\lambda}^{(n,Y)}\frac{t^{n}}{n!},\quad (k \ge 0).
\end{aligned}
\end{equation}
From \eqref{19}, \eqref{33}, \eqref{37}, \eqref{38} and Theorem 2.2, we obtain the following theorem.
\begin{theorem}
For any integers $n \ge k \ge 0$, we have
\begin{align*}
G_{\lambda}^{Y}(n,k)&=B_{n,k}\big(-\beta_{0,\lambda}^{(1,-Y)},-\beta_{1,\lambda}^{(2,-Y)},\dots,-\beta_{n-k,\lambda}^{(n-k+1,-Y)}\big) \\
&=B_{n,k}\big(\beta_{0,-\lambda}^{(1,Y)},\beta_{1,-\lambda}^{(2,Y)},\dots,\beta_{n-k,-\lambda}^{(n-k+1,Y)}\big) \\
&=(-1)^{k}\binom{n-1}{k-1}\beta_{n-k,\lambda}^{(n,-Y)}=(-1)^{k}S_{1,\lambda}^{-Y}(n,k) \\
&=\binom{n-1}{k-1}\beta_{n-k,-\lambda}^{(n,Y)}=S_{1,-\lambda}^{Y}(n,k).
\end{align*}
\end{theorem}

The next two lemmas and Theorem 3.6 can be proved just as those in Lemma 2.1, Lemma 2.2 and Theorem 2.1 of [26]. The details are left to the reader.
\begin{lemma}
 For any integers $n \ge k \ge 0$, we have
\begin{align*}
B_{n,k}&\Big(\frac{E[(Y)_{2,\lambda}]}{2},\frac{E[(Y)_{3,\lambda}]}{3},\dots,\frac{E[(Y)_{n-k+2,\lambda}]}{n-k+2}\Big) \\
&=\sum_{j=0}^{k}\binom{n+k}{k-j}\frac{n!}{(n+k)!}\big(-E[Y]\big)^{k-j}S_{2,\lambda}^{Y}(n+j,j).
\end{align*}
\end{lemma}

\begin{lemma}
For any integer $n \ge 0$, the following identity holds true.
\begin{align*}
\beta_{n,\lambda}^{(\gamma,Y)}=\sum_{k=0}^{n}(-\gamma)_{k}E[Y]^{-\gamma-k} B_{n,k}\Big(\frac{E[(Y)_{2,\lambda}]}{2},\frac{E[(Y)_{3,\lambda}]}{3},\dots,\frac{E[(Y)_{n-k+2,\lambda}]}{n-k+2}\Big). \\
\end{align*}
\end{lemma}
Combining the above two lemmas, we arrive at the next result.
\begin{theorem}
For any integer $n \ge 0$, the following identity holds true.
\begin{equation*}
\beta_{n,\lambda}^{(\gamma,Y)}=\sum_{k=0}^{n}\sum_{j=0}^{k}\binom{\gamma+k-1}{k}\binom{k}{j}\binom{n+j}{j}^{-1}(-1)^{j}E[Y]^{-\gamma-j}S_{2,\lambda}^{Y}(n+j,j).
\end{equation*}
\end{theorem}
Again, the following probabilistic degenerate version of the classical Schl\"omilch formula for the probabilisitc Stirling numbers of both kinds is obtained by proceeding just as in the proof of Theorem 3.1 of [26]. However, for the sake of completeness, we give the short proof. For this, we need the following identities:
\begin{equation} \label{39}
\begin{aligned}
&\sum_{i=j}^{n-k}\binom{n+i-1}{i}\binom{i}{j}=\frac{n}{n+j}\binom{2n-k}{n}\binom{n-k}{j}, \\
&\binom{n-1}{k-1}\binom{2n-k}{n}\frac{n}{n+j}\binom{n-k}{j}\binom{n-k+j}{j}^{-1}=\binom{n+j-1}{n+j-k}\binom{2n-k}{n-k-j}.
\end{aligned}
\end{equation}
The second identity is straightforward. The first identity follows by using $\sum_{i=j}^{m}\binom{i}{j}=\binom{m+1}{j+1}$ and noting that
\begin{equation*}
\sum_{i=j}^{n-k}\frac{n+j}{n}\binom{n+i-1}{i}\binom{i}{j}=\binom{n+j}{j}\sum_{i=n+j-1}^{2n-k-1}\binom{i}{n+j-1}.
\end{equation*}
\begin{theorem}
For any integers $n \ge k \ge 0$, we have
\begin{equation*}
S_{1,\lambda}^{Y}(n,k)=\sum_{j=0}^{n-k}\binom{n+j-1}{n+j-k}\binom{2n-k}{n-k-j}(-1)^{j}E[Y]^{-n-j}S_{2,\lambda}^{Y}(n-k+j,j).
\end{equation*}
\end{theorem}
\begin{proof}
From Theorem 3.5, we see that 
\begin{equation} \label{40}
\begin{aligned}
\beta_{n-k,\lambda}^{(n,Y)}&=\sum_{i=0}^{n-k}\sum_{j=0}^{i}\binom{n+i-1}{i}\binom{i}{j}\binom{n-k+j}{j}^{-1}(-1)^{j}E[Y]^{-n-j}S_{2,\lambda}^{Y}(n-k+j,j) \\
&=\sum_{j=0}^{n-k}\binom{n-k+j}{j}^{-1}(-1)^{j}E[Y]^{-n-j}S_{2,\lambda}^{Y}(n-k+j,j)\sum_{i=j}^{n-k}\binom{n+i-1}{i}\binom{i}{j}.
\end{aligned}
\end{equation}
From Theorem 2.2, \eqref{39} and \eqref{40}, we note
\begin{align*}
S_{1,\lambda}^{Y}(n,k)&=\sum_{j=0}^{n-k}\binom{n-1}{k-1}\binom{2n-k}{n}\frac{n}{n+j}\binom{n-k}{j}\binom{n-k+j}{j}^{-1}\\
&\quad\quad\quad \times (-1)^{j}E[Y]^{-n-j}S_{2,\lambda}^{Y}(n-k+j,j) \\
&=\sum_{j=0}^{n-k} \binom{n+j-1}{n+j-k}\binom{2n-k}{n-k-j}(-1)^{j}E[Y]^{-n-j}S_{2,\lambda}^{Y}(n-k+j,j).
\end{align*}
\end{proof}
From Theorems 3.1 and 3.2, we recall that 
\begin{equation}  \label{41}
\begin{aligned}
&H_{\lambda}^{Y}(n,k)=(-1)^{n}S_{2,\lambda}^{-Y}(n,k)=S_{2,-\lambda}^{Y}(n,k),\\ &G_{\lambda}^{Y}(n,k)=(-1)^{k}S_{1,\lambda}^{-Y}(n,k)=S_{1,-\lambda}^{Y}(n,k).
\end{aligned}
\end{equation}
By Theorem 3.6 and \eqref{41}, we obtain the Schl\"omilch formula for the probabilisitc heterogeneous Stirling numbers of both kinds.
\begin{corollary}
For any integers $n \ge k \ge 0$, we have
\begin{equation*}
G_{\lambda}^{Y}(n,k)=\sum_{j=0}^{n-k}\binom{n+j-1}{n+j-k}\binom{2n-k}{n-k-j}(-1)^{j}E[Y]^{-n-j}H_{\lambda}^{Y}(n-k+j,j).
\end{equation*}
\end{corollary}

\section{Probabilistic degenerate logarithm}

The Stirling numbers of the second kind and those of the first kind are respectively given by
\begin{equation*}
\frac{1}{k!}(e^{t}-1)^{k}=\sum_{n=k}^{\infty}S_{2}(n,k)\frac{t^{n}}{n!}, \quad
\frac{1}{k!}\big(\log(1+t)\big)^{k}=\sum_{n=k}^{\infty}S_{1}(n,k)\frac{t^{n}}{n!},\quad (k \ge 0).
\end{equation*}
The degenerate Stirling numbers of the second kind and those of the first kind are respectively defined by
\begin{equation*}
\frac{1}{k!}(e_{\lambda}(t)-1)^{k}=\sum_{n=k}^{\infty}S_{2,\lambda}(n,k)\frac{t^{n}}{n!}, \,\,
\frac{1}{k!}\big(\log_{\lambda}(1+t)\big)^{k}=\sum_{n=k}^{\infty}S_{1,\lambda}(n,k)\frac{t^{n}}{n!},\,\, (k \ge 0).
\end{equation*}
The probabilistic Stirling numbers of the second kind are specified as
\begin{equation*}
\frac{1}{k!}(E[e^{Yt}]-1)^{k}=\sum_{n=k}^{\infty}S_{2}^{Y}(n,k)\frac{t^{n}}{n!},\quad (k \ge 0).
\end{equation*}
The probabilistic degenerate Stirling numbers of the second kind are introduced as
\begin{equation*}
\frac{1}{k!}(E[e_{\lambda}^{Y}(t)]-1)^{k}=\sum_{n=k}^{\infty}S_{2,\lambda}^{Y}(n,k)\frac{t^{n}}{n!},\quad (k \ge 0).
\end{equation*}
Note that $e^{t}-1$ and $\log(1+t)$ are compositional inverses, as are $e_{\lambda}(t)-1$ and $\log_{\lambda}(1+t)$.
Furthermore, we observe that
\begin{equation*}
\log(1+t)=\sum_{n=1}^{\infty}S_{1}(n,1)\frac{t^{n}}{n!},\quad
\log_{\lambda}(1+t)=\sum_{n=1}^{\infty}S_{1,\lambda}(n,1)\frac{t^{n}}{n!}.
\end{equation*}
Let  
\begin{equation*}
e_{Y}(t)=E[e^{Y t}]-1, \quad e_{Y, \lambda}(t)=E[e_{\lambda}^{Y}(t)]-1. 
\end{equation*}
Then the probabilistic Stirling numbers of the first kind and the probabilisitc degenerate Stirling numbers of the first kind are respectively defined by
\begin{equation*}
\frac{1}{k!}\big(\bar{e}_{Y}(t)\big)^{k}=\sum_{n=k}^{\infty}S_{1}^{Y}(n,k)\frac{t^{n}}{n!},\quad \frac{1}{k!}\big(\bar{e}_{Y,\lambda}(t)\big)^{k}=\sum_{n=k}^{\infty}S_{1,\lambda}^{Y}(n,k)\frac{t^{n}}{n!},\quad (k \ge 0),
\end{equation*}
where $\bar{e}_{Y}(t)$ and $\bar{e}_{Y,\lambda}(t)$ are respectively the compositional inverses of $e_{Y}(t)$ and $e_{Y,\lambda}(t)$. \par
Taking into account these considerations, it is natural to define the {\it{probabilistic logarithm associated with $Y$}}, $\log^{Y}(1+t)$, and the {\it{probabilisitc degenerate logarithm associated with $Y$}}, $\log_{\lambda}^{Y}(1+t)$, respectively as
\begin{equation}\label{42}
\begin{aligned}
&\log^{Y}(1+t)=\bar{e}_{Y}(t)=\sum_{n=1}^{\infty}S_{1}^{Y}(n,1)\frac{t^{n}}{n!}=\sum_{n=1}^{\infty}B_{n-1}^{(n,Y)}\frac{t^{n}}{n!}, \\
&\log_{\lambda}^{Y}(1+t)=\bar{e}_{Y,\lambda}(t)=\sum_{n=1}^{\infty}S_{1,\lambda}^{Y}(n,1)\frac{t^{n}}{n!}=\sum_{n=1}^{\infty}\beta_{n-1,\lambda}^{(n,Y)}\frac{t^{n}}{n!},
\end{aligned} 
\end{equation}
where $B_{n}^{(\gamma,Y)}$ and $\beta_{n,\lambda}^{(\gamma,Y)}$ are respectively as in \eqref{27} and \eqref{22}.
Now, we have
\begin{equation*}
\frac{1}{k!}\big(\log_{\lambda}^{Y}(1+t)\big)^{k}=\sum_{n=k}^{\infty}S_{1,\lambda}^{Y}(n,k)\frac{t^{n}}{n!},\quad (k \ge 0),
\end{equation*}
and
\begin{equation*}
\frac{1}{k!}\big(\log^{Y}(1+t)\big)^{k}=\sum_{n=k}^{\infty}S_{1}^{Y}(n,k)\frac{t^{n}}{n!},\quad (k \ge 0).
\end{equation*}
Recalling that $S_{1}(n,1)=(-1)^{n-1}(n-1)!$ and $S_{1,\lambda}(n,1)=\frac{1}{\lambda}(\lambda)_{n}=(\lambda-1)_{n-1}$, we have
\begin{align*}
&\sum_{n=1}^{\infty}S_{1}(n,1)\frac{t^{n}}{n!}=\sum_{n=1}^{\infty}(-1)^{n-1}(n-1)!\frac{t^{n}}{n!}=\log(1+t), \\
&\sum_{n=1}^{\infty}S_{1,\lambda}(n,1)\frac{t^{n}}{n!}=\frac{1}{\lambda}\sum_{n=1}^{\infty}(\lambda)_{n}\frac{t^{n}}{n!}=\frac{1}{\lambda}\big((1+t)^{\lambda}-1\big)=\log_{\lambda}(1+t),
\end{align*}
as they should be. \par

The probabilistic Stirling numbers of the first kind associated with $Y$, and more generally the probabilistic degenerate Stirling numbers of the first kind associated with $Y$, were previously investigated in [7, 26] and [7], respectively. However, these numbers had not yet been formally integrated into the definitions of the probabilistic logarithm associated with $Y$ and the probabilistic degenerate logarithm associated with $Y$. In this paper, we provide more explicit and tractable definitions for these concepts. Specifically, in Section 5, we derive explicit expressions for the probabilistic degenerate logarithms associated with $Y$ for various discrete and continuous random variables. Furthermore, by applying the probabilistic degenerate version of the Schlömilch formula (see Theorem 3.6) and the definition in \eqref{42}, we show that $\log_{\lambda}^{Y}(1+t)$ can be represented in terms of the probabilistic degenerate Stirling numbers of the second kind associated with $Y$.

\begin{theorem}
The probabilistic degenerate logarithm $\log_{\lambda}^{Y}(1+t)$ is expressed in terms of probabilistic degenerate Stirling numbers of the second kind associated with $Y$ as in the following:
\begin{equation*}
\log_{\lambda}^{Y}(1+t)=\sum_{n=1}^{\infty}\sum_{j=0}^{n-1}\binom{2n-1}{n-1-j}(-1)^{j}E[Y]^{-n-j}S_{2,\lambda}^{Y}(n-1+j,j)\frac{t^{n}}{n!}.
\end{equation*}
\end{theorem}

The Daehee numbers of order $\gamma$ are given by $\big(\frac{\log(1+t)}{t}\big)^{\gamma}=\sum_{n=0}^{\infty}D_{n}^{(\gamma)}\frac{t^{n}}{n!}$, and the degenerate Daehee numbers of order $\gamma$ by $\big(\frac{\log_{\lambda}(1+t)}{t}\big)^{\gamma}=\sum_{n=0}^{\infty}D_{n,\lambda}^{(\gamma)}\frac{t^{n}}{n!}$. We define the {\it{probabilistic degenerate Daehee numbers of order $\gamma$ associated with $Y$}} by
\begin{equation} \label{43}
\Big(\frac{\log_{\lambda}^{Y}(1+t)}{t}\Big)^{\gamma}=\sum_{n=0}^{\infty}D_{n,\lambda}^{(\gamma,Y)}\frac{t^{n}}{n!},
\end{equation}
and the {\it{probabilistic Daehee numbers of order $\gamma$ associated with $Y$}} by
\begin{equation} \label{44}
\Big(\frac{\log^{Y}(1+t)}{t}\Big)^{\gamma}=\sum_{n=0}^{\infty}D_{n}^{(\gamma,Y)}\frac{t^{n}}{n!}.
\end{equation}
We observe that $D_{n,\lambda}^{(\gamma,Y)}$ becomes $D_{n,\lambda}^{(\gamma)}$ when $Y=1$. Furthermore, $D_{n,\lambda}^{(\gamma,Y)} \rightarrow D_{n}^{(\gamma, Y)}$ and $D_{n,\lambda}^{(\gamma)} \rightarrow D_{n}^{(\gamma)}$, as $\lambda \rightarrow 0$. \par
The Cauchy numbers of order $\gamma$ are given by $\big(\frac{t}{\log(1+t)}\big)^{\gamma}=\sum_{n=0}^{\infty}C_{n}^{(\gamma)}\frac{t^{n}}{n!}$, and the degenerate Cauchy numbers of order $\gamma$ by $\big(\frac{t}{\log_{\lambda}(1+t)}\big)^{\gamma}=\sum_{n=0}^{\infty}C_{n,\lambda}^{(\gamma)}\frac{t^{n}}{n!}$. We define the {\it{probabilistic degenerate Cauchy numbers of order $\gamma$ associated with $Y$}} by
\begin{equation} \label{45}
\Big(\frac{t}{\log_{\lambda}^{Y}(1+t)}\Big)^{\gamma}=\sum_{n=0}^{\infty}C_{n,\lambda}^{(\gamma,Y)}\frac{t^{n}}{n!},
\end{equation}
and the {\it{probabilistic Cauchy numbers of order $\gamma$ associated with $Y$}} by
\begin{equation} \label{46}
\Big(\frac{t}{\log^{Y}(1+t)}\Big)^{\gamma}=\sum_{n=0}^{\infty}C_{n}^{(\gamma,Y)}\frac{t^{n}}{n!}.
\end{equation}
We note that $C_{n,\lambda}^{(\gamma,Y)}$ becomes $C_{n,\lambda}^{(\gamma)}$ when $Y=1$. Furthermore, $C_{n,\lambda}^{(\gamma,Y)} \rightarrow C_{n}^{(\gamma,Y)}$ and $C_{n,\lambda}^{(\gamma)} \rightarrow C_{n}^{(\gamma)}$, as $\lambda \rightarrow 0$.
The next four theorems show how the probabilistic degenerate Daehee and Cauchy numbers can be expressed in terms of the probabilisitc degenerate Bernoulli numbers.\par
Replacing $t$ by $E[e_{\lambda}^{Y}(t)]-1$ in \eqref{43}, we get
\begin{equation}
\begin{aligned} \label{47}
\sum_{n=0}^{\infty}\beta_{n,\lambda}^{(\gamma,Y)}\frac{t^{n}}{n!}&=\Big(\frac{t}{E[e_{\lambda}^{Y}(t)]-1}\Big)^{\gamma}=\sum_{k=0}^{\infty}D_{k,\lambda}^{(\gamma,Y)}\frac{1}{k!}\big(E[e_{\lambda}^{Y}(t)]-1\big)^{k} \\
&=\sum_{k=0}^{\infty}D_{k,\lambda}^{(\gamma,Y)}\sum_{n=k}^{\infty}S_{2,\lambda}^{Y}(n,k)\frac{t^{n}}{n!}=\sum_{n=0}^{\infty}\sum_{k=0}^{n}D_{k,\lambda}^{(\gamma,Y)}S_{2,\lambda}^{Y}(n,k) \frac{t^{n}}{n!}.
\end{aligned}
\end{equation}
The next theorem follows from \eqref{47} and by inversion (see Proposition 2.1).
\begin{theorem}
For any integer $n \ge 0$, we have
\begin{equation*}
D_{n,\lambda}^{(\gamma,Y)}=\sum_{k=0}^{n}\beta_{k,\lambda}^{(\gamma,Y)}S_{1,\lambda}^{Y}(n,k).
\end{equation*}
\end{theorem}
In the same manner, one shows the following.
\begin{theorem}
For any integer $n \ge 0$, we have
\begin{equation*}
C_{n,\lambda}^{(\gamma,Y)}=\sum_{k=0}^{n}\beta_{k,\lambda}^{(-\gamma,Y)}S_{1,\lambda}^{Y}(n,k).
\end{equation*}
\end{theorem}
For the next result, we apply the formula (A) with 
\begin{align*}
&g(t)=\bigg(\frac{t}{E[e_{\lambda}^{Y}(t)]-1}\bigg)^{\gamma}=\Big(\frac{t}{e_{Y,\lambda}(t)}\Big)^{\gamma}, \\
&f(t)=e_{Y,\lambda}(t)=E[e_{\lambda}^{Y}(t)]-1.
\end{align*}
Then we note that the left side of the formula (A) is equal to
\begin{equation} \label{48}
[t^{n}]\Big(\frac{\bar{e}_{Y,\lambda}(t)}{t}\Big)^{\gamma}=D_{n,\lambda}^{(\gamma,Y)}\frac{1}{n!},
\end{equation}
while the right hand side of that is equal to
\begin{equation} \label{49}
\begin{aligned}
&\frac{1}{n}[t^{n-1}]\bigg(\bigg(\frac{t}{e_{Y,\lambda}(t)}\bigg)^{\gamma}\bigg)^{\prime}\bigg(\frac{t}{e_{Y,\lambda}(t)}\bigg)^{n} \\
&\quad\quad =\frac{\gamma}{n(n+\gamma)}[t^{n-1}]\bigg(\bigg(\frac{t}{E[e_{\lambda}^{Y}(t)]-1}\bigg)^{n+\gamma}\bigg)^{\prime} \\
&\quad\quad=\frac{\gamma}{n(n+\gamma)}\beta_{n,\lambda}^{(n+\gamma,Y)}\frac{1}{(n-1)!}=\frac{\gamma}{n+\gamma}\frac{1}{n!}\beta_{n,\lambda}^{(n+\gamma,Y)}.
\end{aligned}
\end{equation}
The next result follows from \eqref{48} and \eqref{49}.
\begin{theorem}
For any integer $n \ge 0$ with $\gamma \ne - n$, we have
\begin{equation*}
D_{n,\lambda}^{(\gamma,Y)}=\frac{\gamma}{\gamma+n}\beta_{n,\lambda}^{(n+\gamma,Y)}.
\end{equation*}
\end{theorem}
In the same fashion, one shows the next result.
\begin{theorem}
For any integer $n \ge 0$ with $\gamma \ne n$, we have
\begin{equation*}
C_{n,\lambda}^{(\gamma,Y)}=\frac{\gamma}{\gamma-n}\beta_{n,\lambda}^{(n-\gamma,Y)}.
\end{equation*}
\end{theorem}

\section{Examples}
In this section, we illustrate our results by providing explicit expressions for various discrete and continuous random variables $Y$, as detailed in items (a) through (i). Specifically, for each random variable, we evaluate the following quantities:
\begin{equation*}
E[Y], \quad E[e_{\lambda}^{Y}(t)], \quad S_{2,\lambda}^{Y}(n,k), \quad S_{1,\lambda}^{Y}(n,k), \quad \text{and} \quad \log_{\lambda}^{Y}(1+t).
\end{equation*}
While most of these results were previously established in [7], items (h) and (i) represent new contributions. Notably, the negative binomial random variable $Y \sim NB(r,p)$ in (h) was not addressed in [7]. Furthermore, for the uniform random variable $Y \sim U[0,1]$ in (i), we derive the expression for $S_{1,\lambda}^{Y}(n,k)$ using formulas (A) and (B), as it was absent from the prior study. By utilizing these items, we obtain the following results: \\
$\indent \bullet$ Explicit expressions for $\beta_{n,\lambda}^{(\gamma,Y)}$: Derived from Theorem 3.5 in conjunction with the values for $E[Y]$ and $S_{2,\lambda}^{Y}(n,k)$ provided in (a)–(i). \\
$\indent \bullet$ Schl\"omilch-type formulas: Established via Theorem 3.6 using the expressions for $E[Y]$, $S_{1,\lambda}^{Y}(n,k)$, and $S_{2,\lambda}^{Y}(n,k)$ found in (a)–(i).

\vspace{0.1in}
\noindent (a) Bernoulli random variable $(Y \sim Ber(p))$, with $ 0<p \le 1$: \par
\begin{align*}
&E[Y]= p, \quad   E[e_{\lambda}^{Y}(t)]= 1+p(e_{\lambda}(t)-1),  \\
&S_{2,\lambda}^{Y}(n,k)= p^{k}S_{2,\lambda}(n,k), \quad S_{1,\lambda}^{Y}(n,k)=\frac{1}{p^{n}}S_{1,\lambda}(n,k),\\
&\log_{\lambda}^{Y}(1+t)=\log_{\lambda}\big(1+\frac{t}{p}\big).
\end{align*}

(b) Binomial random variable $(Y \sim b(m,p))$, with $ 0<p \le 1$, $m$ a positive integer: \par
\begin{align*}
&E[Y]= mp, \quad  E[e_{\lambda}^{Y}(t)]=\big(1+p(e_{\lambda}(t)-1)\big)^{m},       \\
&S_{2,\lambda}^{Y}(n,k)=\sum_{j=k}^{n}\sum_{i=j}^{n}m^{j}p^{i}S_{2}(j,k)S_{1}(i,j)S_{2,\lambda}(n,i), \\
&S_{1,\lambda}^{Y}(n,k)= \sum_{j=k}^{n}\sum_{i=j}^{n}\frac{1}{p^{j}}\frac{1}{m^{i}}S_{2}(i,j)S_{1}(n,i)S_{1,\lambda}(j,k), \\
&\log_{\lambda}^{Y}(1+t)=\log_{\lambda}\big(1+\frac{1}{p}\big((1+t)^{\frac{1}{m}}-1\big)\big).
\end{align*}

(c) Poisson random variable $(Y \sim Poisson(\alpha))$, with $ \alpha > 0$: \par
\begin{align*}
&E[Y]= \alpha,  \quad   E[e_{\lambda}^{Y}(t)]=e^{\alpha\big(e_{\lambda}(t)-1\big)},       \\
&S_{2,\lambda}^{Y}(n,k)= \sum_{j=k}^{n}\alpha^{j}S_{2}(j,k)S_{2,\lambda}(n,j), \quad S_{1,\lambda}^{Y}(n,k)= \sum_{j=k}^{n}\frac{1}{\alpha^{j}}S_{1,\lambda}(j,k)S_{1}(n,j), \\
&\log_{\lambda}^{Y}(1+t)=\log_{\lambda}\big(1+\frac{1}{\alpha}\log(1+t)\big).
\end{align*}

(d) Exponential random variable $(Y \sim Exp(\alpha))$, with $\alpha > 0$: \par
\begin{align*}
&E[Y]=\frac{1}{\alpha},  \quad   E[e_{\lambda}^{Y}(t)]= \frac{\alpha}{\alpha-\log e_{\lambda}(t)}, \\
&S_{2,\lambda}^{Y}(n,k)= \sum_{j=k}^{n}\binom{j}{k}(j-1)_{j-k}\frac{1}{\alpha^{j}}\lambda^{n-j}S_{1}(n,j),\\       
&S_{1,\lambda}^{Y}(n,k)= \sum_{j=k}^{n}(-1)^{n-j}L(n,j)\alpha^{j}\lambda^{j-k}S_{2}(j,k), \\
&\log_{\lambda}^{Y}(1+t)=\log_{\lambda}\big(e^{\alpha \frac{t}{1+t}}\big).
\end{align*}

(e) Gamma random variable $(Y \sim \Gamma(\alpha, \beta))$, with $\alpha,\, \beta >0$: \par
\begin{align*}
&E[Y]=\frac{\alpha}{\beta},  \quad  E[e_{\lambda}^{Y}(t)]= \bigg(\frac{\beta}{\beta-\frac{1}{\lambda}\log(1+\lambda t)}\bigg)^{\alpha},  \\
&S_{2,\lambda}^{Y}(n,k)= \sum_{l=0}^{n}\sum_{j=0}^{k}(-1)^{k-j}\frac{1}{k!}\binom{k}{j}(\alpha j+l-1)_{l}\frac{1}{\beta^{l}}\lambda^{n-l}S_{1}(n,l), \\     
&S_{1,\lambda}^{Y}(n,k)= \sum_{l=k}^{\infty}\sum_{j=0}^{l}(-1)^{j}\frac{1}{l!}\binom{l}{j}\Big(\frac{-j}{\alpha}\Big)_{n}\beta^{l}\lambda^{l-k}S_{2}(l,k), \\
&\log_{\lambda}^{Y}(1+t)=\log_{\lambda}\Big(e^{\beta\big(1-(1+t)^{-\frac{1}{\alpha}}\big)}\Big).
\end{align*}

(f) Geometric random variable $(Y \sim G(p))$, with $0<p<1$: \par
\begin{align*}
&E[Y]= \frac{1}{p},  \quad   E[e_{\lambda}^{Y}(t)]= \frac{pe_{\lambda}(t)}{1-(1-p)e_{\lambda}(t)},    \\
&S_{2,\lambda}^{Y}(n,k)= \frac{1}{k!}\Big(\frac{1}{p-1}\Big)^{k}\sum_{j=0}^{k}\binom{k}{j}(-1)^{j}h_{n,\lambda}^{(j)}\Big(\frac{1}{1-p}\Big), \\  
&S_{1,\lambda}^{Y}(n,k)=\sum_{j=k}^{n}L(n,j)p^{j}(p-1)^{n-j}S_{1,\lambda}(j,k), \\
&\log_{\lambda}^{Y}(1+t)=\log_{\lambda}\Big(\frac{1+t}{1+(1-p)t}\Big),
\end{align*}
where $h_{n,\lambda}^{(r)}(u)$ are the degenerate Frobenius-Euler numbers of order $r$, given by
\begin{equation*}
\Big(\frac{1-u}{e_{\lambda}(t)-u} \Big)^{r}=\sum_{n=0}^{\infty} h_{n,\lambda}^{(r)}(u)\frac{t^{n}}{n!}, \quad (u \ne 1). 
\end{equation*}

(g) Normal random variable with parameters $(\mu, \sigma^{2})$ $(Y \sim N(\mu, \sigma^{2})$, with $\mu \ne 0, \sigma>0$: \par
\begin{align*}
&E[Y]=\mu,  \quad   E[e_{\lambda}^{Y}(t)]= e^{\mu\log e_{\lambda}(t)+\frac{1}{2}\sigma^{2}\big(\log e_{\lambda}(t)\big)^{2}}, \\
&S_{2,\lambda}^{Y}(n,k)=\sum_{m=k}^{n}\sum_{j=k}^{m}\frac{m!}{j!}\binom{j}{m-j}\mu^{2j-m}\Big(\frac{\sigma^{2}}{2}\Big)^{m-j}\lambda^{n-m}S_{2}(j,k)S_{1}(n,m),\\
&S_{1,\lambda}^{Y}(n,k)=\frac{1}{\lambda^{k}}\sum_{m=0}^{n}\sum_{j=k}^{\infty}\sum_{l=0}^{j}(-1)^{j+l}\frac{1}{j!}\Big(\frac{\lambda \mu}{\sigma^{2}} \Big )^{j}\Big(\frac{l}{2}\Big)_{m}2^{m}\Big(\frac{\sigma}{\mu} \Big)^{2m}\binom{j}{l}S_{2}(j,k) S_{1}(n,m),\\
&\log_{\lambda}^{Y}(1+t)
=\frac{1}{\lambda}\sum_{n=1}^{\infty}\sum_{m=0}^{n}\sum_{j=1}^{\infty}\sum_{l=1}^{j}(-1)^{j+l}\frac{1}{j!}\binom{j}{l}\Big(\frac{\lambda \mu}{\sigma^{2}}\Big)^{j}\Big(\frac{l}{2}\Big)_{m}2^{m}\Big(\frac{\sigma}{\mu}\Big)^{2m}S_{1}(n,m)\frac{t^{n}}{n!}.
\end{align*}
(h) Negative binomial random variable $(Y \sim NB(r,p))$, with $0<p<1$, $r$ a positive integer. This is not treated in [7]. The expression for $E[Y]$ is well known and those for $E[e_{\lambda}^{Y}(t)]$ and $\log_{\lambda}^{Y}(1+t)$ are  easy to see. We show here only the expression for $S_{1,\lambda}^{Y}(n,k)$, leaving the proof for $S_{2,\lambda}^{Y}(n,k)$ to the reader.\par
\begin{align*}
&E[Y]= \frac{r(1-p)}{p},  \quad   E[e_{\lambda}^{Y}(t)]= \bigg(\frac{p}{1-(1-p)e_{\lambda}(t)} \bigg)^{r},\\
&S_{2,\lambda}^{Y}(n,k)= \sum_{j=0}^{\infty}\sum_{m=0}^{j}\sum_{l=0}^{m}\frac{1}{(k-l)!}(p-1)^{j}(p^{r}-1)^{k-l}p^{rl}(-r)^{m}\frac{(j)_{n,\lambda}}{j!}S_{1}(j,m)S_{2}(m,l), \\
&S_{1,\lambda}^{Y}(n,k)= \sum_{l=0}^{k}\sum_{m=l}^{\infty}\frac{(-p)^{m}}{(k-l)!}\frac{\big(-\frac{m}{r}\big)_{n}}{m!}\Big(\frac{1}{1-p}\Big)^{\lambda l}\Big(\log_{\lambda}\Big(\frac{1}{1-p}\Big)\Big)^{k-l}S_{1,\lambda}(m,l),\\
&\log_{\lambda}^{Y}(1+t)=\log_{\lambda}\Big(\frac{1}{1-p}\big(1-p(1+t)^{-\frac{1}{r}} \big)\Big).
\end{align*}
 To derive the expression for $S_{1,\lambda}^{Y}(n,k)$, we first observe that 
\begin{equation*}
\log_{\lambda}(AB)=A^{\lambda}\log_{\lambda}(B)+\log_{\lambda}(A).
\end{equation*}
Here, for $e_{Y,\lambda}(t)=E[e_{\lambda}^{Y}(t)]-1=\big(\frac{p}{1-(1-p)e_{\lambda}(t)} \big)^{r}-1$, it is immediate to see that
\begin{align*}
\bar{e}_{Y,\lambda}(t)=\log_{\lambda}^{Y}(1+t)=\log_{\lambda}\Big(\frac{1}{1-p}\big(1-p(1+t)^{-\frac{1}{r}} \big)\Big).
\end{align*}
Now, we proceed as follows:
\begin{align*}
\frac{1}{k!}&\big(\bar{e}_{Y,\lambda}(t)\big)^{k}=\frac{1}{k!}\Big(\log_{\lambda}\Big(\frac{1}{1-p}\Big)+\Big(\frac{1}{1-p}\Big)^{\lambda}\log_{\lambda}(1-p(1+t)^{-\frac{1}{r}}\big)\Big)^{k} \\
&=\frac{1}{k!}\sum_{l=0}^{k}\binom{k}{l}\Big(\log_{\lambda}\Big(\frac{1}{1-p}\Big)\Big)^{k-l}\Big(\frac{1}{1-p}\Big)^{\lambda l}l!\frac{1}{l!}\Big(\log_{\lambda}\big(1-p(1+t)^{-\frac{1}{r}}\big)\Big)^{l} \\
&=\frac{1}{k!}\sum_{l=0}^{k}\binom{k}{l}\Big(\log_{\lambda}\Big(\frac{1}{1-p}\Big)\Big)^{k-l}\Big(\frac{1}{1-p}\Big)^{\lambda l}l! \\
&\quad\quad\quad\quad \times \sum_{m=l}^{\infty}S_{1,\lambda}(m,l)\frac{1}{m!}(-p)^{m}\sum_{n=0}^{\infty}\Big(-\frac{1}{r}\Big)^{n}\langle m \rangle_{n,r}\frac{t^{n}}{n!} \\
&=\sum_{n=0}^{\infty}\sum_{l=0}^{k}\sum_{m=l}^{\infty}\frac{(-p)^{m}}{(k-l)!}\frac{\big(-\frac{m}{r}\big)_{n}}{m!}\Big(\frac{1}{1-p}\Big)^{\lambda l}\Big(\log_{\lambda}\Big(\frac{1}{1-p}\Big)\Big)^{k-l}S_{1,\lambda}(m,l) \frac{t^{n}}{n!}.
\end{align*}

(i) Uniform random variable $(Y \sim U[0,1])$: \par
\begin{align*}
&E[Y]=\frac{1}{2}, \quad   E[e_{\lambda}^{Y}(t)]=\frac{\lambda}{\log(1+\lambda t)}\big(e_{\lambda}(t)-1\big)=\frac{1}{\log e_{\lambda}(t)}\big(e_{\lambda}(t)-1\big),  \\
&S_{2,\lambda}^{Y}(n,k)= \frac{1}{k!}\sum_{m=0}^{n}\sum_{j=0}^{k}\binom{k}{j}\binom{m+j}{j}^{-1}(-1)^{k-j}\lambda^{n-m}S_{2}(m+j,j)S_{1}(n,m).      
\end{align*}
The above quantities are known from [7]. Here we find an explicit expression for $S_{1,\lambda}^{Y}(n,k)$ by using Lagrange inversion formulas, from which an expression for $\log_{\lambda}^{Y}(1+t)$ follows as well.
For this, we first show the following lemma.
\begin{lemma}
For $Y \sim U[0,1]$ and $n \ge k+1$, we have
\begin{equation*}
[t^{n}]\frac{1}{k!}\big(\bar{e}_{Y,\lambda}(t)\big)^{k}=\frac{1}{(n-k)k!}[t^{n-k-1}]\bigg(\frac{t}{\frac{e^{t}-1}{t}-1}\bigg)^{n-k}\frac{d}{dt}\bigg(\frac{\log_{\lambda}(e^{t})}{\frac{e^{t}-1}{t}-1} \bigg)^{k},
\end{equation*}
where $e_{Y,\lambda}(t)=E[e_{\lambda}^{Y}(t)]-1=\frac{e_{\lambda}(t)-1}{\log e_{\lambda}(t)}-1=\frac{e^{\log e_{\lambda}(t)}-1}{\log e_{\lambda}(t)}-1$.
\begin{proof}
By the formula (B), we have
\begin{equation} \label{50}
[t^{n}]\big(\bar{e}_{Y,\lambda}(t)\big)^{k}=\frac{k}{n}[t^{n-k}]\Big(\frac{t}{e_{Y,\lambda}(t)}\Big)^{n},
\end{equation}
where $\big(\frac{t}{e_{Y,\lambda}(t)}\big)^{n}=g(\bar{f}(t))$, \\
with
\begin{equation*}
g(t)=\bigg(\frac{\log_{\lambda}(e^{t})}{\frac{e^{t}-1}{t}-1}\bigg)^{n}=\bigg(\frac{t^{2}}{e^{t}-1-t}\frac{e^{\lambda t}-1}{\lambda t}\bigg)^{n},\quad \bar{f}(t)=\log e_{\lambda}(t).
\end{equation*}
Note here that $f(t)=\log_{\lambda}e^{t}$. Using the formula (A), we see that
\begin{equation} \label{51}
\begin{aligned} 
[t^{m}]g(\bar{f}(t))&=\frac{1}{m}[t^{m-1}]g^{\prime}(t)\Big(\frac{t}{f(t)}\Big)^{m} \\
&=\frac{1}{m}[t^{m-1}]n\bigg(\frac{\log_{\lambda}(e^{t})}{\frac{e^{t}-1}{t}-1}\bigg)^{n-m-1}\bigg(\frac{\log_{\lambda}(e^{t})}{\frac{e^{t}-1}{t}-1}\bigg)^{\prime}\bigg(\frac{t}{\frac{e^{t}-1}{t}-1}\bigg)^{m} \\
&=\frac{n}{m(n-m)}[t^{m-1}]\bigg(\bigg(\frac{\log_{\lambda}(e^{t})}{\frac{e^{t}-1}{t}-1}\bigg)^{n-m}\bigg)^{\prime}\bigg(\frac{t}{\frac{e^{t}-1}{t}-1}\bigg)^{m}
\end{aligned}
\end{equation}
Applying \eqref{51} with $m=n-k$ to \eqref{50}, we get the desired result.
\end{proof}
\end{lemma}
Before proceeding further, we recall special cases of Bernoulli-Pad\'e numbers $A_{k,n}$ studied by Howard (see [10]), which are defined by:
\begin{equation*}
\frac{\frac{1}{k!}t^{k}}{e^{t}-\sum_{j=0}^{k-1}\frac{t^{j}}{j!}}=\sum_{n=0}^{\infty}A_{k,n}\frac{t^{n}}{n!}, \quad (A_{k,0}=1).
\end{equation*}
Then it is immediate to see that
\begin{align} 
&\bigg(\frac{t^{2}}{e^{t}-1-t}\bigg)^{n}=2^{n}\sum_{m=0}^{\infty}\sum_{k_1+k_2+\cdots+k_n=m}\binom{m}{k_1,k_2,\dots,k_n}A_{2,k_1}A_{k_2}\cdots A_{k_n}\frac{t^{m}}{m!}, \label{52} \\
&\bigg(\bigg(\frac{t^{2}}{e^{t}-1-t}\bigg)^{n}\bigg)^{\prime}=2^{n}\sum_{m=0}^{\infty}\sum_{k_1+k_2+\cdots+k_n=m+1}\binom{m+1}{k_1,k_2,\dots,k_n}A_{2,k_1}A_{k_2}\cdots A_{k_n}\frac{t^{m}}{m!}. \label{53}
\end{align}
We also need 
\begin{equation} \label{54}
\frac{e^{\lambda t}-1}{\lambda t}=\sum_{n=0}^{\infty}\frac{\lambda^{n}}{n+1}\frac{t^{n}}{n!}, \quad
\bigg(\frac{e^{\lambda t}-1}{\lambda t}\bigg)^{\prime}=\sum_{n=0}^{\infty}\frac{\lambda^{n+1}}{n+2}\frac{t^{n}}{n!}.
\end{equation}

\begin{theorem}
For $Y \sim U[0,1]$, we have
\begin{equation} \label{55}
\begin{aligned}
\frac{1}{k!}&\big(\log_{\lambda}^{Y}(1+t)\big)^{k}=\sum_{n=k}^{\infty}S_{1,\lambda}^{Y}(n,k)\frac{t^{n}}{n!}=\sum_{n=k}^{\infty}\frac{2^{n}}{n}\binom{n}{k}\sum_{m=0}^{n-k}(n-m)\binom{n-k}{m}\\
& \times \sum_{j_1+j_2+\cdots+j_n=m}\binom{m}{j_1,j_2,\dots,j_n}A_{2,j_1}A_{2,j_2}\cdots A_{2,j_n} \\
&\times \sum_{l_1+l_2+\cdots+l_k=n-k-m}\binom{n-k-m}{l_1,l_2,\dots,l_k}\frac{\lambda^{n-k-m}}{(l_1+1)(l_2+1)\cdots(l_k+1)}\frac{t^{n}}{n!},
\end{aligned}
\end{equation}
where the number $A_{2,n}$'s are given by
\begin{equation*}
\frac{\frac{1}{2!}t^{2}}{e^{t}-1-t}=\sum_{n=0}^{\infty}A_{2,n}\frac{t^{n}}{n!}.
\end{equation*}
\begin{proof}
Here we only give the proof for $k=1$. Applying the above lemma with $k=1$ and using \eqref{52}-\eqref{54}, for $n \ge 2$, we have
\begin{equation} \label{56}
\begin{aligned}
&[t^{n}]\bar{e}_{Y,\lambda}(t)=\frac{1}{n-1}[t^{n-2}]\bigg(\frac{t}{\frac{e^{t}-1}{t}-1}\bigg)^{n-1}\frac{d}{dt}\bigg(\frac{\log_{\lambda}(e^{t})}{\frac{e^{t}-1}{t}-1} \bigg)\\
&=\frac{1}{n-1}[t^{n-2}]\bigg(\frac{t^{2}}{e^{t}-1-t}\bigg)^{n-1} \frac{d}{dt}\bigg(\frac{t^{2}}{e^{t}-1-t}\frac{e^{\lambda t}-1}{\lambda t}\bigg) \\
&=\frac{1}{n-1}[t^{n-2}]\bigg(\frac{t^{2}}{e^{t}-1-t}\bigg)^{n-1}\\
&\quad\times\bigg\{\bigg(\frac{t^{2}}{e^{t}-1-t}\bigg)^{\prime}\bigg(\frac{e^{\lambda t}-1}{\lambda t}\bigg)+\bigg(\frac{t^{2}}{e^{t}-1-t}\bigg)\bigg(\frac{e^{\lambda t}-1}{\lambda t}\bigg)^{\prime}\bigg\}\\
&=\frac{1}{n-1}[t^{n-2}]\bigg\{ \frac{1}{n}\bigg(\bigg(\frac{t^{2}}{e^{t}-1-t}\bigg)^{n}\bigg)^{\prime}\bigg(\frac{e^{\lambda t}-1}{\lambda t}\bigg)+\bigg(\frac{t^{2}}{e^{t}-1-t}\bigg)^{n}\bigg(\frac{e^{\lambda t}-1}{\lambda t}\bigg)^{\prime} \bigg\} \\
&=\frac{2^{n}}{n!}\sum_{m=0}^{n-2}\binom{n-2}{m}\sum_{j_1+j_2+\cdots+j_n=m+1}\binom{m+1}{j_1,j_2,\dots,j_n}A_{2,j_1}A_{2,j_2}\cdots A_{2,j_n}\frac{\lambda^{n-m-2}}{n-m-1} \\
&+\frac{2^{n}}{(n-1)!}\sum_{m=0}^{n-2}\binom{n-2}{m}\sum_{j_1+j_2+\cdots+j_n=m}\binom{m}{j_1,j_2,\dots,j_n}A_{2,j_1}A_{2,j_2}\cdots A_{2,j_n}\frac{\lambda^{n-m-1}}{n-m}\\
&=\frac{2^{n}}{n!}\sum_{m=1}^{n-1}\binom{n-2}{m-1}\sum_{j_1+j_2+\cdots+j_n=m}\binom{m}{j_1,j_2,\dots,j_n}A_{2,j_1}A_{2,j_2}\cdots A_{2,j_n}\frac{\lambda^{n-m-1}}{n-m} \\
&+\frac{2^{n}}{(n-1)!}\sum_{m=0}^{n-2}\binom{n-2}{m}\sum_{j_1+j_2+\cdots+j_n=m}\binom{m}{j_1,j_2,\dots,j_n}A_{2,j_1}A_{2,j_2}\cdots A_{2,j_n}\frac{\lambda^{n-m-1}}{n-m}\\
&=\frac{2^{n}}{n!}\sum_{m=0}^{n-1}\binom{n-1}{m}\sum_{j_1+j_2+\cdots+j_n=m}\binom{m}{j_1,j_2,\dots,j_n}A_{2,j_1}A_{2,j_2}\cdots A_{2,j_n}\lambda^{n-m-1}.
\end{aligned}
\end{equation}
Now, the result follows from \eqref{56} by observing that $[t] \bar{e}_{Y,\lambda}(t)=[t^{0}]\big(\frac{t}{e_{Y,\lambda}(t)}\big)=2 $, using the formula (C).
\end{proof}
\end{theorem}

Letting $\lambda \rightarrow 0$ in \eqref{55}, we obtain the follwoing.
\begin{corollary}
For $Y \sim U[0,1]$, we have
\begin{align*}
&\frac{1}{k!}\big(\log^{Y}(1+t)\big)^{k}=\sum_{n=k}^{\infty}S^{Y}(n,k)\frac{t^{n}}{n!} \\
&=k\sum_{n=k}^{\infty}\frac{2^{n}}{n}\binom{n}{k}\sum_{j_1+j_2+\cdots j_n=n-k}\binom{n-k}{j_1,j_2,\dots,j_n} A_{2,j_1}A_{2,j_2}\cdots A_{2,j_n}\frac{t^{n}}{n!}.
\end{align*}
\end{corollary}
Taking $k=1$ in \eqref{55}, we get the follwing expression for the logarithm. 
\begin{corollary}
For $Y \sim U[0,1]$, we have
\begin{align*}
&\log_{\lambda}^{Y}(1+t)=\sum_{n=1}^{\infty}S_{1,\lambda}^{Y}(n,1)\frac{t^{n}}{n!}=\sum_{n=1}^{\infty}2^{n}\sum_{m=0}^{n-1}\binom{n-1}{m} \\
&\times \sum_{j_1+j_2+\cdots+j_n=m}\binom{m}{j_1,j_2,\dots,j_n}A_{2,j_1}A_{2,j_2}\cdots A_{2,j_n}\lambda^{n-m-1}\frac{t^n}{n!}.
\end{align*}
\end{corollary}

\section{Conclusion}
The study of special polynomials and numbers has recently shifted toward two significant frontiers: degenerate versions, originating from the work of Carlitz, and probabilistic extensions, which associate classical sequences with random variables $Y$. In this paper, we have successfully bridged gaps in the existing literature regarding these sequences. By defining the probabilistic (discrete) logarithms $\log^{Y}(1+t)$ and $\log_{\lambda}^{Y}(1+t)$, we provided a consistent framework for generating $S_{1}^{Y}(n,k)$ and $S_{1,\lambda}^{Y}(n,k)$ that preserves orthogonality and inverse relations. Furthermore, our introduction of probabilistic degenerate Daehee and Cauchy numbers offers a more robust model than previous versions, as they maintain a dependency on the random variable $Y$ even at $x=0$. Finally, the derivation of the probabilistic degenerate Schlömilch formula and the exploration of heterogeneous Stirling numbers of the first kind complete a comprehensive structural symmetry between classical and probabilistic combinatorics. These results offer promising tools for further application in modeling stochastic systems and exploring new identities within umbral calculus.


\begin{thebibliography}{9} 
\bibitem{1}
Adell, J. A. \emph{Probabilistic Stirling numbers of the second kind and applications,} J. Theor. Probab. \textbf{35} (2022), 636–652. 
\bibitem{2} 
Adell, J. A.; B{\'e}nyi, B. \emph{Probabilistic Stirling numbers and applications,} Aequat. Math. \textbf{98} (2024), 1627–1646.
\bibitem{3}
Adell, J. A.; Lekuona, A. \emph{A probabilistic generalization of the Stirling numbers of the second kind,} J. Number Theory \textbf{194} (2019), 225-355.
\bibitem{4}
Adell, J. A.; Lekuona, Alberto. \emph{Explicit expressions for a certain subset of Appell polynomials: a probabilistic perspective,} Integers \textbf{21} (2021), Paper No. A60, 10 pp.
\bibitem{5}
Carlitz, L. \emph{Degenerate Stirling, Bernoulli and Eulerian numbers,} Utilitas Math. \textbf{15} (1979), 51-–88.
\bibitem{6}
Comtet, L. \emph{Advanced combinatorics,} The art of finite and infinite expansions, Revised and enlarged edition. D. Reidel Publishing Co., Dordrecht, 1974.
\bibitem{7}
Kim, D. S.; Kim, T. \emph{Probabilistic Stirling and degenerate Stirling numbers,} arXiv:2511.14568 [math.NT]
\bibitem{8}
Kim, D. S.; Kim, T. \emph{Degenerate Sheffer sequences and $\lambda$-Sheffer sequences,} J. Math. Anal. Appl. \textbf{493} (2021), 124521.
\bibitem{9}
Kim, D. S.; Kim, T. \emph{Stirling numbers associated with sequences of polynomials,} Appl. Comput. Math. \textbf{22} (2023), no. 1, 80-115.
\bibitem{10}
Kim, D. S.; Kim, T. \emph{An umbral calculus approach to Bernoulli-Pad\'e polynomials,} In Mathematical Analysis, Approximation Theory and Their Applications, edited by Themistocles M. Rassias and Vijay Gupta, 363-382, Springer, 2016.
\bibitem{11}
Kim, D. S.; Kim, T.; Jang, G. W. \emph{A note on degenerate Stirling numbers of the first kind,} Proc. Jangjeon Math. Soc. \textbf{21} (2018), 393–404.
\bibitem{12}
Kim, T. \emph{A note on degenerate Stirling polynomials of the second kind,} Proc. Jangjeon Math. Soc. \textbf{20}  (2017),  no. 3, 319-331.
\bibitem{13}
Kim, T.; Kim, D. S. \emph{Degenerate Laplace transform and degenerate gamma function,} Russ. J. Math. Phys. \textbf{24} (2017), no. 2, 241-248.
\bibitem{14}
Kim, T.; Kim, D. S. \emph{Note on the degenerate gamma function,} Russ. J. Math. Phys. \textbf{27} (2020), no. 3, 352–358.
\bibitem{15}
Kim, T.; Kim, D. S. \emph{Probabilistic Bernoulli and Euler polynomials,} Russ. J. Math. Phys. \textbf{31} (2024), no. 1, 94-105.
\bibitem{16} 
Kim, T.; Kim, D. S. \emph{Probabilistic degenerate Stirling numbers of the first kind and their applications,} Eur. J. Math. \textbf{10} (2024), Article No. 73.
\bibitem{17}
Kim, T.; Kim, D. S. \emph{Heterogeneous Stirling numbers and heterogeneous Bell polynomials,} Russ. J. Math. Phys. \textbf{32} (2025), no. 3, 498–509.
\bibitem{18}
Kim, T.; Kim, D. S. \emph{Probabilistic heterogeneous Stirling numbers and Bell polynomials,} arXiv:2601.09964 [math.NT]
\bibitem{19}
Kim, T.; Kim, D. S.; Kwon, H.-I. \emph{A note on degenerate Stirling numbers and their applications,} Proc. Jangjeon Math. Soc. \textbf{21} (2018), no. 2, 195-203.
\bibitem{20}
Kim, T.; Kim, D. S.; Kwon, J. \emph{Probabilistic degenerate Stirling polynomials of the second kind and their applications,} Math. Comput. Model. Dyn. Syst. \textbf{30} (2024), no. 1, 16-30.
\bibitem{21}
Luo, L.; Kim, T.; Kim, D. S.; Ma, Y. \emph{Probabilistic degenerate Bernoulli and degenerate Euler polynomials,} Math. Comput. Model. Dyn. Syst. \textbf{30} (2024), no. 1, 342-363.
\bibitem{22}
Park, J.-W. \emph{A generalization of the degenerate Daehee polynomials,} J. Inequal. Spec. Func. \textbf{16} (2025), no. 3, 1-11. https://doi.org/10.54379/jiasf-2025-3-1
\bibitem{23}
Quaintance, J.; Gould, H. W. \emph{Combinatorial identities for Stirling numbers, The unpublished  notes of H. W. Gould, With a foreword by George E. Andrews,} World Scientific Publishing Co. Pte. Ltd., Singapore, 2016.
\bibitem{24}
Roman, S. \emph{The umbral calculus,} Pure and Applied Mathematics, 111.Academic Press, Inc. [Harcourt Brace Jovanovich, Publishers], New York, 1984. 
\bibitem{25}
Ross, S. M. \emph{Introduction to probability models,} Thirteenth edition, Academic Press, London, 2024. 
\bibitem{26}
Xu, A. \emph{Computing a family of probabilistic numbers in terms of probabilistic Stirling numbers of the second kind,} Appl. Math. Sci. Eng. \textbf{33} (2025), no. 1, 2485250.
\bibitem{27}
Yun, S. J.; Park, J.-W. \emph{On a generation of degenerate Daehee polynomials,} AIMS Math. \textbf{10} (2025), no. 5, 12286–12298. DOI: 10.3934/math.2025556























\end{thebibliography}
\end{document}